\documentclass[10pt, a4paper, twoside]{amsart}
\usepackage{amsthm}
\usepackage{amsmath}
\usepackage{amssymb}
\usepackage{amscd}
\usepackage{mathtools}
\usepackage[all]{xy}
\usepackage{indentfirst}
\usepackage{comment}
\usepackage{microtype}
\usepackage{tikz}
\usepackage{tikz-cd}
\usetikzlibrary{patterns}
\usepackage{enumitem}

\usepackage{framed}
\usepackage{color}
\definecolor{shadecolor}{gray}{0.875}
\definecolor{col}{RGB}{42, 95, 151}
\usepackage[colorlinks=true, linkcolor=col, citecolor=col, filecolor = col, menucolor = col, urlcolor = col]{hyperref}

\newtheorem{thm}{\bf Theorem}[section]

\newtheorem{lem}[thm]{\bf Lemma}

\newtheorem{cor}[thm]{\bf Corollary}

\newtheorem*{thm*}{\bf Theorem}
\newtheorem*{cor*}{\bf Corollary}

\theoremstyle{definition}

\newtheorem{rem}[thm]{\it Remark}

\newtheorem*{df*}{\bf Definition}
\newtheorem*{not*}{\bf Notation}
\newtheorem*{dfs*}{\bf Definitions}
\newtheorem*{ack*}{\bf Acknowledgements}
\newtheorem*{dfrem*}{\bf Definition and Remark}

\def\P{\mathbb{P}}
\def\C{\mathbb{C}}
\def\F{\mathbb{F}}
\def\Q{\mathbb{Q}}

\def\Z{\mathbb{Z}}

\def\X{\mathcal{X}}
\def\Y{\mathcal{Y}}
\def\O{\mathcal{O}}

\DeclareMathOperator{\Proj}{Proj}

\DeclareMathOperator{\Ker}{Ker}

\DeclareMathOperator{\Alb}{Alb}

\DeclareMathOperator{\Coker}{Coker}

\DeclareMathOperator{\tors}{tors}

\DeclareMathOperator{\alg}{alg}
\DeclareMathOperator{\cl}{cl}

\DeclareMathOperator{\Top}{top}

\makeindex

\title[A pencil of Enriques surfaces]
{A pencil of Enriques surfaces with non-algebraic integral Hodge classes}
\date{\today}
\author{John Christian Ottem}
\address{Department of Mathematics, University of Oslo, Box 1053, Blindern, 0316 Oslo, Norway}
\email{johnco@math.uio.no}
\author{Fumiaki Suzuki}
\address{Department of Mathematics, University of Illinois at Chicago, 851 S. Morgan Street, Chicago, IL, USA}
\email{fsuzuk2@uic.edu}

\begin{document}

\begin{abstract}
We prove that there exists a pencil of Enriques surfaces defined over $\Q$ with non-algebraic integral Hodge classes of non-torsion type.
This gives the first example of a threefold with the trivial Chow group of zero-cycles on which the integral Hodge conjecture fails.
As an application, we construct a fourfold which gives the negative answer to a classical question of Murre on the universality of the Abel-Jacobi maps in codimension three.
\end{abstract}
\maketitle

\section{Introduction}
For a smooth complex projective variety $X$, we denote by $CH^{p}(X)$ the Chow group of codimension $p$ cycles and by $H^{2p}(X,\Z)$ the Betti cohomology group of degree $2p$.
The image $H^{2p}_{\alg}(X,\Z)\subseteq H^{2p}(X,\Z)$ of the cycle class map $\cl^{p}\colon CH^{p}(X)\rightarrow H^{2p}(X,\Z)$ is contained in the group $Hdg^{2p}(X,\Z)\subseteq H^{2p}(X,\Z)$ of integral Hodge classes.
The integral Hodge conjecture is the statement that these two subgroups of $H^{2p}(X,\Z)$ coincide. 
While this statement holds for $p=0, 1$ and $\dim X$, it is known that it can fail in general. 
The first counterexample was given by Atiyah-Hirzebruch \cite{AH}, who  constructed a projective manifold admitting a non-algebraic degree four torsion class. 
Later, a different type of counterexample was constructed by Koll\'ar \cite[p. 134, Lemma]{BCC}, who proved that for certain high degree hypersurfaces $X\subset \P^4$, the generator of $H^4(X,\Z)=\Z$ is not algebraic.
This means that the natural inclusion
\[
H^{4}_{\alg}(X,\Z)/\tors \subset Hdg^{4}(X,\Z)/\tors
\]
can be strict.
Since then, many other examples of non-algebraic integral Hodge classes have been found, both of torsion type \cite{SV,BO} and of non-torsion type \cite{CTV,T,D}.

In this paper, we study Enriques surface fibrations over curves and show that they can admit non-algebraic integral Hodge classes of non-torsion type.
\begin{thm}[=Theorem \ref{t'}]\label{t}
There exists a pencil of Enriques surfaces defined over $\Q$ such that the cohomology groups $H^{i}(X,\Z)$ are torsion-free for all $i$ and the  inclusion
\[
H^{4}_{\alg}(X,\Z) \subsetneq Hdg^{4}(X,\Z)
\]is strict.
\end{thm}
One can compare Theorem \ref{t} with the result of Benoist--Ottem \cite{BO}, which showed that the integral Hodge conjecture can fail on products $S\times C$ for an Enriques surface $S$ and curve $C$ of genus at least one. 
In those examples, the non-algebraic classes in question are 2-torsion, but the integral Hodge classes are algebraic modulo torsion classes by the K\"unneth formula.

Theorem \ref{t} also relates to certain questions concerning rational points of algebraic varieties.
In a letter to Grothendieck, Serre asked whether a projective variety over the function field of a curve always has a rational point 
if it is  $\mathcal{O}$-acyclic, that is, $H^{i}(X,\mathcal{O}_{X})=0$ for all $i>0$.
This question was answered negatively by Grabber--Harris--Mazur--Starr \cite{GHMS}, who constructed an Enriques surface without rational points over the function field of a complex curve.
Later, more explicit constructions of such Enriques surfaces were given by Lafon \cite{La} and Starr \cite{St}.

According to \cite{St}, Esnault expected that the Enriques surfaces of \cite{GHMS} and \cite{La} would 
satisfy a stronger property that every closed point has even degree over the base field. 
If that were the case, it would give a pencil of Enriques surfaces with non-algebraic integral Hodge classes of non-torsion type (this follows from \cite[Theorem 7.6]{CTV}).
In fact, this observation was the starting point of the present paper.

Another feature of our example is that it has a trivial Chow group of zero-cycles. Indeed, Bloch--Kas--Lieberman \cite{BKL} proved that 
$CH_0(S)=\Z$ for any Enriques surface $S$, and from this one deduces that the same holds for any pencil of Enriques surfaces (see Lemma \ref{l2}).
To our knowledge, this is the first example of a threefold with the trivial Chow group of zero-cycles on which the integral Hodge conjecture fails
(see \cite[Subsection 5.7]{CTV} for a threefold constructed by Colliot-Th\'el\`ene and Voisin which conjecturally satisfies this condition).
We emphasize that it is not a priori obvious that such a threefold should exist.
For instance, typical examples with the trivial Chow groups of zero-cycles are given by rationally connected varieties
while the integral Hodge conjecture holds on rationally connected threefolds by a result of Voisin \cite{V1}.

 As an application, we settle a classical question on the universality of the Abel-Jacobi maps.
We denote by $A^{p}(X)\subset CH^{p}(X)$ the subgroup of cycle classes algebraically equivalent to zero.
The Abel-Jacobi map
\[
\psi^{p}\colon A^{p}(X)\rightarrow J^{p}_{a}(X),
\]
where $J^{p}_{a}(X)$ is the Lieberman intermediate Jacobian \cite{Lie}, is regular:
it defines an invariant on $A^{p}(X)$ with values in an abelian variety such that for any algebraic family of codimension $p$ cycles on $X$,
the function mapping each point of the base of the family to the value of the invariant of the corresponding codimension $p$ cycle is algebraic (see Section $4$ for a more precise definition of regular homomorphisms).
A classical question of Murre \cite[Section 7]{M3}\cite[p. 132]{GMV} asks whether the Abel-Jacobi map is universal among all regular homomorphisms
(see \cite{V2} for another universality question from a different perspective).
This is known to be true for $p=1,2$ and $\dim X$ \cite{M2}.
Combined with \cite[Theorem 1.3]{S}, Theorem \ref{t} implies that the question has a negative answer in the first open case $p=3$.

\begin{cor}[=Corollary \ref{c'}]\label{c}
Let $X$ be the pencil of Enriques surfaces of Theorem \ref{t}.
Then there exists an elliptic curve $E$ such that the Abel-Jacobi map
\[
\psi^{3}\colon A^{3}(X\times E)\rightarrow J_{a}^{3}(X\times E)
\]
is not universal: it factors through a universal regular homomorphism and the projection is an isogeny with non-zero kernel.
\end{cor}

This paper is organized as follows.
In Section $2$, we study the geometry of the pencils of Enriques surfaces appearing in Theorem \ref{t}. These are defined as the rank one degeneracy loci of maps of vector bundles on $\P^{1}\times \P^{2}\times \P^{2}$. 
In particular, we compute their integral cohomology groups and Chow groups of zero-cycles. In Section $3$, we prove the main theorem, using a specialization argument. 
In Section $4$, we apply the main theorem to Murre's question on the universality of the Abel-Jacobi maps.

We work over the complex numbers throughout.

\begin{ack*}
The authors would like to thank Olivier Benoist, J\o rgen Vold Rennemo, Jason Starr, and Claire Voisin for interesting discussions. 
The second author wishes to thank his advisor Lawrence Ein for constant support and warm encouragement.
JCO was supported by the Research Council of Norway project no. 250104. 
FS was supported by the NSF Grant No. DMS-1801870.
This project started while the authors were in residence 
at the Mathematical Sciences Research Institute in Berkeley, California, during the Spring 2019 semester.
\end{ack*}
 
\section{Geometry of pencils of Enriques surfaces}

In this paper, a {\em pencil of Enriques surfaces} will mean a smooth complex threefold $\X$ with a fibration $\X\to \P^{1}$ over $\P^{1}$ whose general fibers are Enriques surfaces. In the course of the proof of Theorem \ref{t}, we will give a few explicit constructions of such threefolds. We start with the construction of the Enriques surfaces themselves. 

We will fix the following notation\footnote{We use Grothendieck's notation for projective bundles: for a vector bundle $\mathcal E$, $\P(\mathcal E)$ paramterizes one-dimensional quotients of $\mathcal E$.}:
\begin{itemize}[leftmargin=*, label={-}]
\item $\P_{A}=\P_{\P^{2}\times \P^{2}}(\mathcal{O}(2,0)\oplus \mathcal{O}(0,2))$, $E_{1}=\P_{\P^{2}\times \P^{2}}(\mathcal{O}(2,0))$, $E_{2}=\P_{\P^{2}\times \P^{2}}(\mathcal{O}(0,2))$
\item $\P_{B}=\P_{\P^{2}\times \P^{2}}(\mathcal{O}(1,0)\oplus \mathcal{O}(0,1))$, $F_{1}=\P_{\P^{2}\times \P^{2}}(\mathcal{O}(1,0))$, $F_{2}=\P_{\P^{2}\times \P^{2}}(\mathcal{O}(0,1))$
\item $\P_{C}=\P(H^{0}(\P_{B},\mathcal{O}(1)))$, $P_{1}=\P(H^{0}(\P^{2}\times \P^{2},\mathcal{O}(1,0)))$, $P_{2}=\P(H^{0}(\P^{2}\times \P^{2},\mathcal{O}(0,1)))$.
\end{itemize}
These spaces are related as follows. We can regard $P_{1}$ and $P_{2}$ as disjoint planes in the five-dimensional projective space $\P_{C}$
 via the idetification $$H^{0}(\P_{B},\mathcal{O}(1))=H^{0}(\P^{2}\times \P^{2},\mathcal{O}(1,0))\oplus H^{0}(\P^{2}\times \P^{2},\mathcal{O}(0,1)).$$ Then the projective bundle $\P_{B}$ is identified with the blow-up of $\P_{C}$ along the union of $P_{1}$ and $P_{2}$ with the exceptional divisors $F_{1}$ and $F_{2}$. Moreover, there is a natural involution $\iota$ on $\P_{C}$ induced by the involution on $H^{0}(\P_{B},\mathcal{O}(1))$ 
with the ($\pm 1$)-eigenspaces $H^{0}(\P^{2}\times \P^{2},\mathcal{O}(1,0))$ and $H^{0}(\P^{2}\times \P^{2}, \mathcal{O}(0,1))$, respectively.
The involution $\iota$ lifts to an involution on $\P_{B}$, and we have $\P_{A}=\P_{B}/\iota$.  Thus there is a double cover $\P_{B}\rightarrow \P_{A}$ over $\P^{2}\times \P^{2}$, which is ramified along $F_{i}$, and  the divisors $F_{i}$ are mapped isomorphically onto $E_{i}$ for $i=1,2$. 

The projective models of the Enriques surfaces are defined as follows. 
On $\P^{2}\times \P^{2}$, we consider a map of vector bundles
\[
u\colon \mathcal{O}^{\oplus 3}\rightarrow \mathcal{O}(2,0)\oplus \mathcal{O}(0,2).
\]
Let $X$ be the rank one degeneracy locus of $u$.

\begin{lem}\label{lES}
If $u$ is general, then $X$ is an Enriques surface.
\end{lem}

\begin{proof}
Since the vector bundle $\mathcal{O}(2,0)\oplus \mathcal{O}(0,2)$ is globally generated, $X$ is smooth of dimension two by the Bertini theorem for degeneracy loci. 

To show that $X$ is an Enriques surface, we will describe its K3 cover $Y$.
The map $u$ defines a global section $s$ of $\mathcal{O}(1)^{\oplus 3}$ on the projective bundle $\P_{A}$. When $u$ is generic, the zero set  $Z(s)\subset \P_A$ maps isomorphically onto $X$ via the bundle projection $\P_{A}\rightarrow \P^{2}\times \P^{2}$.

On the other hand, the map $u$ also defines a global section of $\mathcal{O}(2)^{\oplus 3}$ on $\P_{B}$ invariant under the action of $\iota$.
Indeed, as $(q_{*}\mathcal{O}_{\P_{B}}(2))^{\iota}=(q_{*}q^{*}\mathcal{O}_{\P_{A}}(1))^{\iota}=\mathcal{O}_{\P_{A}}(1)$, where $q\colon \P_{B}\rightarrow \P_{A}=\P_{B}/\iota$ is a natural projection,
we have a natural identification
\[
H^{0}(\P_{B},\mathcal{O}(2))^{\iota}=H^{0}(\P_{A}, \mathcal{O}(1))=H^{0}(\P^{2}\times \P^{2},\mathcal{O}(2,0)\oplus \mathcal{O}(0,2)).
\]
Let $Y\subset \P_B$ denote the zero set of this section. When $u$ is general, we have $X\cap E_{i}=Y\cap F_{i}=\emptyset$, so $Y$ maps isomorphically to a smooth intersection of three quadrics in $\P_C$ via the blow-down map $\P_B\to \P_C$. In particular, $Y$ is a K3 surface. Again since $Y\cap F_{i}=\emptyset$, the composition $\P_{B}\rightarrow \P_A\to \P^{2}\times \P^{2}$ restricts to an \'etale double cover $Y\rightarrow X$. Hence $X$ is an Enriques surface.
\end{proof}

\begin{rem}\label{r}
The proof of Lemma \ref{lES} shows that the construction of Enriques surfaces introduced above coincides with a classical one from \cite[Example VIII.18]{B}.
\end{rem}

We will now use a variant of the above construction to construct pencils of Enriques surfaces.
On $\P^{1}\times \P^{2}\times \P^{2}$, we consider a map of vector bundles
\[
v\colon \mathcal{O}^{\oplus 3} \rightarrow \mathcal{O}(1,2,0)\oplus \mathcal{O}(1,0,2).
\]
Let $\mathcal{X}$ be the rank one degeneracy locus of $v$.

\begin{lem}\label{l1}
If $v$ is general, then $\mathcal{X}$ is a pencil of Enriques surfaces by the first projection $\mathcal{X}\rightarrow \P^{1}$.
Moreover, we have $H^{i}(\mathcal{X},\mathcal{O}_{\mathcal{X}})=0$ for all $i>0$.
\end{lem}
\begin{proof}
Since the vector bundle $\O(1,2,0)\oplus \O(1,0,2)$ is globally generated, $\mathcal{X}$ is smooth and $\dim \mathcal{X}=3$ by the Bertini theorem for degeneracy loci. 
Moreover, $\mathcal{X}$ is connected since it is defined by three equations of tridegree $(2,2,2)$.
The resolution of the ideal sheaf $\mathcal{I}_{\mathcal{X}}$ of $\mathcal{X}$ in $\P^{1}\times \P^{2}\times \P^{2}$ has the form
\[
0\to \O(-3,-4,-2)\oplus \O(-3,-2,-4) \to \O(-2,-2,-2)^{\oplus3}\to \mathcal{I}_{\mathcal{X}}\to 0.
\]
From this it follows that $H^{i}(\mathcal{X},\mathcal{O}_{\mathcal{X}})=0$ for all $i>0$.
\end{proof}

We assume that $v$ is general in what follows.

\begin{lem}\label{l2}
The degree homomorphism $\deg \colon CH_{0}(\mathcal{X})\rightarrow \Z$ is an isomorphism.
\end{lem}
\begin{proof}
Let $C\subset \mathcal{X}$ be a smooth curve which is a complete intersection of very ample divisors.
Then $CH_{0}(\mathcal{X})$ is supported on $C$.
This follows from the fact that 
any class in $CH_{0}(\mathcal{X})$ is represented by a zero-cycle supported on a union of smooth fibers of the first projection $\mathcal{X}\rightarrow \P^{1}$ by the moving lemma,
and that the Chow group of zero-cycles on any given Enriques surface is trivial due to Bloch--Kas--Lieberman \cite{BKL}.

We consider a natural homomorphism $\phi\colon \Ker(\deg)\rightarrow \Alb(\mathcal{X})$ induced by the Albanese map.
Since $CH_{0}(\mathcal{X})$ is supported on a curve, the decomposition of the diagonal \cite{BS} implies that $\Ker(\phi)$ is torsion.
Moreover $\Ker(\phi)$ is torsion-free by the Roitman theorem \cite{R}.
Hence we have $\Ker(\phi)=0$ and $\phi$ is an isomorphism. 
In our situation, $\Alb(\mathcal{X})=0$ since $H^{1}(\mathcal{X},\mathcal{O}_{\mathcal{X}})=0$ by Lemma \ref{l1}. 
Therefore $\Ker(\deg)=0$.
The proof is complete.
\end{proof}

To study the geometric properties of the threefold $\mathcal{X}$ in more detail, it will be convenient to involve its double cover. Recalling the construction above, we get a diagram
\[
\xymatrix{
\P^{1}\times \P_{B} \ar[rr] \ar[d] && \P^{1}\times \P_{A} \ar[d] \\
\P^{1}\times \P_{C}                 && \P^{1}\times \P^{2}\times \P^{2}         
},
\]
where $\P^{1}\times \P_{B}\rightarrow \P^{1}\times \P_{A}$ is the quotient by the involution $\iota$ (which acts as before on $\P_{B}$ and as the identity on the first factor)
and $\P^{1}\times \P_{B}\rightarrow \P^{1}\times \P_{C}$ is the blow-up of $\P^{1}\times \P_{C}$ along the union of $\P^{1}\times P_{1}$ and $\P^{1}\times P_{2}$.
Restricting to $\mathcal{X}$, we get the following diagram
\[
\xymatrix{
\mathcal{Y} \ar[rr] \ar[d] && \mathcal{X}' \ar[d]^{\simeq} \\
\mathcal{Y}_{\min}            && \mathcal{X}        
}.
\]
The varieties appearing in this diagram can be described as follows.
The map $v$ induces a global section of $\mathcal{O}(1,1)^{\oplus 3}$ on $\P^{1}\times \P_{A}$
as well as global sections of $\mathcal{O}(1,2)^{\oplus 3}$ on $\P^{1}\times \P_{B}$ and $\P^{1}\times \P_{C}$ which are invariant under the action of $\iota$;
the varieties $\mathcal{X}', \mathcal{Y}, \mathcal{Y}_{\min}$ are the zero sets of these sections.
By generality, $\mathcal{X}',\mathcal{Y},\mathcal{Y}_{\min}$ are smooth threefolds;
$\mathcal{X}'$ is mapped isomorphically onto $\mathcal{X}$, so we can identify $\mathcal{X}'$ with $\mathcal{X}$;
$\mathcal{Y}$ is a double cover of $\mathcal{X}'=\mathcal{X}$; and $\mathcal{Y}_{\min}$ is a minimal model of $\mathcal{Y}$. Note that $\Y$ and $\Y_{\min}$ are K3 surface fibrations via the first projection.

An easy computation shows that each of the intersections $\mathcal{Y}_{\min}\cap (\P^{1}\times P_{i})$ consists of twelve points $y_{i,1},\cdots, y_{i,12}$.
Then the map $\mathcal{Y}\rightarrow \mathcal{Y}_{\min}$ is the blow-up of $\Y_{\min}$ along $y_{i,j}$ whose exceptional divisors $F_{i,j}$ are the components of $\mathcal{Y}\cap (\P^{1}\times F_{i})$.
Moreover the double cover $\mathcal{Y}\rightarrow \mathcal{X}$ is ramified along $F_{i,j}$ which are mapped isomorphically onto $E_{i,j}$, the components of $\mathcal{X}\cap (\P^{1}\times E_{i})$.

\begin{lem}
The threefold $\mathcal X$ has Kodaira dimension one.
\end{lem}
\begin{proof}
Let $X$ be the class of a fiber of the first projection $\mathcal{X}\rightarrow \P^{1}$.
It is straightforward to compute that
\[
2K_{\mathcal{X}}=2X+ \sum_{i=1}^{2}\sum_{j=1}^{12}E_{i,j}.
\]
As the normal bundles $N_{E_{i,j}/\mathcal{X}}=\mathcal{O}_{\P^{2}}(-2)$ are negative, we obtain that $\kappa(\mathcal{X})=1$.
\end{proof}

\begin{lem}
The Hodge numbers of $\mathcal{X}$ are given by $h^{0,0}(\mathcal{X})=h^{3,3}(\mathcal{X})=1$, $h^{1,1}(\mathcal{X})=h^{2,2}(\mathcal{X})=26$, $h^{1,2}(\mathcal{X})=h^{2,1}(\mathcal{X})=45$, and $h^{p,q}(\mathcal{X})=0$ otherwise.
\end{lem}
\begin{proof}
We first compute the Picard number $\rho(\mathcal{X})$.
Using the Lefschetz hyperplane section theorem, $\mathcal{Y}_{\min}$ has Picard number two, so $\rho(\mathcal{Y})=\rho(\mathcal{Y}_{\min})+24=26$.
Moreover, the action of $\iota$ on the Picard group of $\mathcal{Y}$ is trivial,
so also $\rho(\mathcal{X})=26$.

We next compute the Betti numbers $b_{i}(\mathcal{X})$.
It is straightforward to compute the topological Euler characteristic $\chi_{\Top}(\mathcal{X})=c_{3}(T_{\mathcal{X}})=-36$.
Obviously $b_{0}(\mathcal{X})=b_{6}(\mathcal{X})=1$.
Moreover, $b_{1}(\mathcal{X})=b_{5}(\mathcal{X})=0$ and $b_{2}(\mathcal{X})=b_{4}(\mathcal{X})=\rho(\mathcal{X})=26$ using Lemma \ref{l1}.
Therefore $b_{3}(\mathcal{X})=90$. 

Now the computation of the Hodge numbers are immediate using Lemma \ref{l1} again.
\end{proof}

We next study the topology of $\mathcal X$. We fix the following notation:
\begin{itemize}[leftmargin=*, label={-}]
\item $\mathcal{X}_{\min}=\mathcal{Y}_{\min}/\iota$;
\item $\mathcal{Y}^{\circ}=\mathcal{Y}_{\min}- \left\{y_{i,j}\right\}_{i,j}$;
\item $\mathcal{X}^{\circ}=\mathcal{Y}^{\circ}/\iota$;
\item $V_{i,j}\subset \mathcal Y$, a small ball around $y_{i,j}$;
\item $U_{i,j}=V_{i,j}/\iota$.
\end{itemize}
We have $\mathcal{Y}_{\min}=\mathcal{Y}^{\circ}\cup \left(\bigcup_{i,j}V_{i,j}\right)$ and $\mathcal{X}_{\min}=\mathcal{X}^{\circ}\cup \left(\bigcup_{i,j}U_{i,j}\right)$.

\begin{lem}\label{l3}
The threefold $\mathcal{X}$ is simply connected, and the cohomology groups $H^{i}(\mathcal{X},\Z)$ are torsion-free for all $i$. 
\end{lem}

\begin{proof}
By the universal coefficient theorem, it is enough to prove that $\pi_{1}(\mathcal{X})=0$ and $H^{3}(\mathcal{X},\Z)$ is torsion-free.

We first prove that $\pi_{1}(\mathcal{X})=0$.
We have a natural pushout diagram
\[
\xymatrix{
\pi_1(U_{i,j}\cap \mathcal{X}^{\circ}) \ar[r] \ar[d] & \pi_1(\mathcal{X}^{\circ}) \ar[d] \\
\pi_1(U_{i,j}) \ar[r]                 & \pi_1(\mathcal{X}_{\min})          
}.
\]
By Lefschetz, $\mathcal Y$ and hence $\mathcal{Y}^{\circ}$ is simply connected. So since the quotient map $\pi\colon\mathcal{Y}^{\circ}\to \mathcal{X}^{\circ}$ is \'etale, we have $\pi_1(\mathcal{X}^{\circ})=\Z/2$.
The neighbourhood $U_{i,j}\subset \mathcal X$ is homotopic to the affine cone over a Veronese surface, so we have $\pi_1(U_{i,j})=0$. 
Finally, since the map $V_{i,j}\cap \mathcal{Y}^{\circ} \to U_{i,j}\cap \mathcal{X}^{\circ}$ is homotopic to the universal covering map $(\C^3-0)\to (\C^3-0)/\pm$, we have $\pi_1(U_{i,j}\cap \mathcal{X}^{\circ})=\Z/2$. In fact, this cover is induced by the restriction of $\pi$ to $V_{i,j}\cap \mathcal{Y}^\circ$, so the map $\pi_1(U_{i,j}\cap \mathcal{X}^{\circ})\to \pi_1(\mathcal{X}^{\circ})$ is non-zero, hence an isomorphism. 
From the pushout diagram above, we then get $\pi_1(\mathcal{X}_{\min})=0$. 
Resolving a finite cyclic quotient singularity does not change the fundamental group (\cite[Theorem 7.8]{K}), so we also get $\pi_1(\mathcal{X})=0$.

We next prove that $H^{3}(\mathcal{X},\Z)$ is torsion-free.
The long exact sequence for cohomology groups with supports gives
\[
\bigoplus_{i,j}H_{E_{i,j}}^{3}(\mathcal{X},\Z)\rightarrow H^{3}(\mathcal{X},\Z)\rightarrow H^{3}(\mathcal{X}^{\circ},\Z).
\]
Since $H_{E_{i,j}}^{3}(\mathcal{X},\Z)=H_{3}(E_{i,j},\Z)=0$, the group $H^{3}(\mathcal{X},\Z)$ injects into $H^{3}(\mathcal{X}^{\circ},\Z)$.
In particular, we are reduced to showing that $H^{3}(\mathcal{X}^{\circ},\Z)$ is torsion-free.

Since $\mathcal{X}^{\circ}$ is the quotient of $\mathcal{Y}^{\circ}$ by the group $\langle \iota\rangle\simeq \Z/2$, we can apply the Cartan--Leray spectral sequence 
\[
E_2^{p,q}=H^p(\Z/2,H^q(\mathcal{Y}^{\circ},\Z))\Rightarrow H^{p+q}(\mathcal{X}^{\circ},\Z)
\]
to compute the cohomology groups of $\mathcal{X}^{\circ}$.
We need to compute $H^{q}(\mathcal{Y}^{\circ},\Z)$ for $0\leq q\leq 3$ and the action of $\iota$ on these groups.
Since $\mathcal{Y}^{\circ}$ is obtained from $\mathcal{Y}_{\min}$ by removing finitely many points,
we have an identification $H^{q}(\mathcal{Y}^{\circ},\Z)=H^{q}(\mathcal{Y}_{\min},\Z)$.
Clearly $H^{0}(\mathcal{Y}_{\min},\Z)=\Z$.
By the Lefschetz hyperplane theorem, $H^{1}(\mathcal{Y}_{\min},\Z)=0$, and the groups $H^{2}(\mathcal{Y}_{\min},\Z)$ and $H^3(\mathcal{Y}_{\min},\Z)$ are torsion-free. 
Moreover, the action of $\iota$ on $H^{q}(\mathcal{Y}_{\min},\Z)$ is trivial for $0\leq q \leq 2$.
Since the group cohomology $H^{p}(\Z/2,\Z)=0$ for $p$ odd, it follows that $E^{p,3-p}_{2}=0$ for $p\neq 0$.
Therefore there is an injection
\[H^{3}(\mathcal{X}^{\circ},\Z)\hookrightarrow E^{0,3}_{2}=H^{0}(\Z/2, H^{3}(\mathcal{Y}^{\circ},\Z))=H^{3}(\mathcal{Y}^{\circ},\Z)^{\iota},
\]
where the right hand side is torsion-free.
This completes the proof.
\end{proof}

\section{Proof of Theorem \ref{t}}
We are now ready to prove our main result:
\begin{thm}\label{t'}
There exists a map of vector bundles on $\P^{1}\times \P^{2}\times \P^{2}$
\[
\mathcal{O}^{\oplus 3}\rightarrow \mathcal{O}(1,2,0)\oplus \mathcal{O}(1,0,2)
\]
defined over $\Q$ such that the rank one degeneracy locus $\mathcal{X}$ is a pencil of Enriques surfaces 
such that the cohomology groups $H^{i}(\mathcal{X},\Z)$ are torsion-free for all $i$ and there is a strict inclusion
\[
H^{4}_{\alg}(\mathcal{X},\Z) \subsetneq Hdg^{4}(\mathcal{X},\Z).
\]
\end{thm}
\begin{proof}
We set $\P^{1}\times \P^{2}\times \P^{2}=\Proj\C[S,T]\times \Proj\C[X_{0},X_{1},X_{2}]\times \Proj\C[Y_{0},Y_{1},Y_{2}]$.
Fix a sufficiently large prime number $p$. We consider a map of vector bundles as above given by the matrix
\[
M=
\left(
\begin{array}{ccc}
P_{1} & Q_{1} & R_{1}\\
SP_{2}+pP_{3}& SQ_{2}+pQ_{3}& SR_{2}+pR_{3}
\end{array}
\right),
\]
where $P_{1},Q_{1}, R_{1}$ (resp. $P_{2},Q_{2},R_{2}$; $P_{3}, Q_{3}, R_{3}$) are general tri-homogeneous polynomials of tri-degree $(1,2,0)$ (resp. $(0,0,2)$; $(1,0,2)$) over $\Q$.
The degeneracy locus $\mathcal{X}$ is a pencil of Enriques surfaces defined by the $2\times2$-minors of $M$.
The torsion-freeness of the cohomology groups follows from Lemma \ref{l3}, so
it remains to prove that the integral Hodge conjecture does not hold on $\mathcal{X}$.

The closed subscheme defined by $P_{1}=Q_{1}=R_{1}=0$ is a disjoint union of twelve components $E_{1,1},\ldots, E_{1,12}$ isomorphic to $\P^{2}$. We note that this union is defined over $\Q$, even though each $E_{i,j}$ may not be.
First we prove that for a given algebraic one-cycle $\alpha$ on $\mathcal{X}$, we have
\begin{eqnarray}
\deg(\alpha/\P^{1})\equiv \alpha\cdot \left(\sum_{j=1}^{12}E_{1,j}\right) \mod 2.
\end{eqnarray}
We use a specialization argument.
We spread out $\mathcal{X}_{\overline{\Q}}$ over a valuation ring $R$ with the maximal ideal containing $p$.
The ideal of the flat closure of $\mathcal{X}_{\overline{\Q}}$ in $(\P^{1}\times \P^{2}\times \P^{2})_{R}$ is generated by the $2\times 2$-minors of $M$ and
\[
F=\det
\left(
\begin{array}{ccc}
P_{1} & Q_{1} & R_{1}\\
P_{2} & Q_{2} & R_{2}\\
P_{3} & Q_{3} & R_{3}
\end{array}
\right).
\]
The specialization over $\overline{\F}_{p}$ consists of two components: one is a pencil of Enriques surfaces $\widetilde{\mathcal{X}}_{0}$ defined by the $2\times 2$-minors of the matrix
\[
N=\left(
\begin{array}{ccc}
P_{1} & Q_{1} & R_{1}\\
P_{2}& Q_{2} & R_{2}
\end{array}
\right);
\]
the other is defined by $S=F=0$. 
It is straightforward to check that $\widetilde{\mathcal{X}}_{0}$ is smooth. 

The closed subscheme defined by $P_{1}=Q_{1}=R_{1}=0$ is again a disjoint union of twelve components $E_{1,1},\ldots, E_{1,12}$ isomorphic to $\P^{2}$
 and disjoint from the fiber over $S=0$ by the generality of $P_{1},Q_{1},R_{1}$.
 We prove that for a given one-cycle $\alpha_{0}$ on the specialization over $\overline{\F}_{p}$, we have
 \begin{eqnarray}
\deg(\alpha_{0}/\P^{1})\equiv \alpha_{0}\cdot \left(\sum_{j=1}^{12}E_{1,j}\right) \mod 2.
\end{eqnarray}
We may assume that $\alpha_{0}$ is supported on $\widetilde{\mathcal{X}}_{0}$. Let $D_{1}$ be the Cartier divisor on $\widetilde{\mathcal{X}}_{0}$ defined by $P_{1}=0$.
Since $D_{1}$ is of type $(1,2,0)$, we have
\[
\deg (\alpha_{0}/\P^{1})\equiv \alpha_{0}\cdot D_{1} \mod 2.
\]
On the other hand, we have
\[
D_{1}=D_{2}+\sum_{j=1}^{12} E_{1,j},
\]
where $D_{2}$ is the Cartier divisor on $\widetilde{\mathcal{X}}_{0}$ defined by $P_{2}=0$.
Indeed, expanding the $2\times 2$-minors of $N$, it is easily seen that the identity holds on each of the open subsets $P_{2}, Q_{2}, R_{2}\neq 0$;
these open subsets form an open cover of $\widetilde{\mathcal{X}}_{0}$ by the generality of $P_{2},Q_{2},R_{2}$.
Since $D_{2}$ is of type $(0,0,2)$, we have
\[
\alpha_{0}\cdot D_{1} \equiv \alpha_{0}\cdot \left(\sum_{j=1}^{12}E_{1,j}\right) \mod 2.
\]
The congruence (2) follows, so does the congruence (1) by the specialization homomorphism \cite[Section 20.3]{F}.

The Hodge structure of $H^{4}(\mathcal{X},\Z)$ is trivial since we have $H^{2}(\mathcal{X},\mathcal{O}_{\mathcal{X}})=0$ by Lemma \ref{l1}.
The proof of the theorem is reduced to proving that there exists a class $\beta \in H^{4}(\mathcal{X},\Z)=H_{2}(\mathcal{X},\Z)$ such that 
\[\deg(\beta/\P^{1})=\pm 1, \,\beta\cdot \left(\sum_{j=1}^{12}E_{1,j}\right)=0;
\] 
such $\beta$ is not algebraic according to the congruence (1).
Since $E_{1,1},\ldots, E_{1,12}$ are the images of $F_{1,1},\ldots, F_{1,12}$ under the double cover $\mathcal{Y}\rightarrow \mathcal{X}$,
 it is enough to prove that there exists $\gamma \in H^{4}(\mathcal{Y},\Z)=H_{2}(\mathcal{Y},\Z)$ such that
\[
\deg(\gamma/\P^{1})=\pm 1, \, \gamma \cdot \left(\sum_{j=1}^{12}F_{1,j}\right)=0;
\]
the class $\beta$ will be the push-forward of $\gamma$.
By the Lefschetz hyperplane section theorem, the push-forward $H_{2}(\mathcal{Y}_{\min},\Z)\rightarrow H_{2}(\P^{1},\Z)$ is surjective.
Let $\gamma_{\min}\in H^{4}(\mathcal{Y}_{\min},\Z)=H_{2}(\mathcal{Y}_{\min},\Z)$ be an element mapped to a generator of $H_{2}(\P^{1},\Z)$.
Then the pullback $\gamma\in H^{4}(\mathcal{Y},\Z)$ of $\gamma_{\min}$ satisfies the desired property.
The proof is complete.
\end{proof}

\begin{rem}
The specialization used in the proof of Theorem \ref{t'} deserves a few more comments.
The specialization consists of two components: $\widetilde{\mathcal{X}}_{0}$ defined by the $2\times2$-minors of $N$, and $R$ defined by $S=F=0$.
The component $\widetilde{\mathcal{X}}_{0}$ is smooth, and it is a pencil of Enriques surfaces by the first projection $\widetilde{\mathcal{X}}_{0}\rightarrow \P^{1}$. 
On the other hand, $R$ has isolated singularities, and a smooth model $\overline{R}$ of $R$ is another pencil of Enriques surfaces
with a small contraction $\overline{R}\rightarrow R$ contracting $\P^{1}$s over the singular points of $R$.
In addition, $\widetilde{\mathcal{X}}_{0}$ and $R$ intersect in a fiber over $S=0$, and the intersection is an Enriques surface $Z$ in $\P^{2}\times \P^{2}$.

Remarkably, both of the components $\widetilde{\mathcal{X}}_{0}$ and $R$ are rationally connected:
the projections 
\[
\widetilde{\mathcal{X}}_{0}\hookrightarrow \P^{1}\times \P^{2}\times \P^{2}\xrightarrow{pr_{2}} \P^{2}, \,\, R\hookrightarrow \P^{1}\times\P^{2}\times \P^{2}\xrightarrow{pr_{3}}\P^{2}
\] 
are conic bundles, therefore this follows from \cite[Corollary 1.3]{GHS}.
In particular, the integral Hodge conjecture holds on $\widetilde{\mathcal{X}}_{0}$ and $\overline{R}$ by a result of Voisin \cite{V1}.
As a consequence, $H_{2}(\widetilde{\mathcal{X}}_{0},\Z)$ and $H_{2}(R,\Z)$ are generated by algebraic cycles.

It turns out, however, that this is not the case for the union $\widetilde{\mathcal{X}}_{0}\cup R$.
A key point here is the subtle difference between the Mayer-Vietoris sequence for homology groups and Chow groups.
For the homology groups, we have an exact sequence
\[
H_{2}(\widetilde{\mathcal{X}}_{0},\Z)\oplus H_{2}(R,\Z)\rightarrow H_{2}(\widetilde{\mathcal{X}}_{0}\cup R, \Z)\rightarrow H_{1}(Z,\Z)=\Z/2\rightarrow 0.
\]
For the Chow groups, on the other hand, we obviously have a surjection
\[
CH_{1}(\widetilde{\mathcal{X}}_{0})\oplus CH_{1}(R)\twoheadrightarrow CH_{1}(\widetilde{\mathcal{X}}_{0}\cup R)
\]
(see also \cite[Example 1.8.1]{F}).
It follows that $H_{2}(\widetilde{\mathcal{X}}_{0}\cup R, \Z)$ is not generated by algebraic cycles.
\end{rem}

A small modification of the above arguments yields a generalization of Theorem \ref{t'} to higher dimensions:
\begin{thm}
For a given positive integer $n$,
there exists a map of vector bundles on $\P^{1}\times \P^{2n}\times \P^{2n}$
\[
\mathcal{O}^{\oplus (2n+1)}\rightarrow \mathcal{O}(1,2,0)\oplus \mathcal{O}(1,0,2)
\]
defined over $\Q$ such that the rank one degeneracy locus $\mathcal{X}$ is a smooth $(2n+1)$-fold
with a fibration over $\P^{1}$ whose general fibers are $2n$-folds $X$ with $H^{i}(X,\mathcal{O}_{X})=0$ for all $i>0$ and universal Calabi-Yau double covers $Y\rightarrow X$
such that
\begin{enumerate}
\item[(i)] $H^{i}(\mathcal{X},\mathcal{O}_{\mathcal{X}})=0$ for all $i>0$;
\item[(ii)] $\kappa(\mathcal{X})=1$;
\item[(iii)] $\mathcal{X}$ is simply connected, and the cohomology group $H^{3}(\mathcal{X},\Z)$ is torsion-free;
\item[(iv)] the inclusion $H_{2,\alg}(\mathcal{X},\Z) \subsetneq Hdg_{2}(\mathcal{X},\Z)$ is strict.
\end{enumerate}
\end{thm}

\section{Application to the universality of the Abel-Jacobi maps}
Let $V$ be a smooth complex projective variety. 
For an integer $p$, we let $A^{p}(V)\subseteq CH^p(V)$ denote the subgroup of cycles algebraically equivalent to zero. 
We recall that a homomorphism $\phi\colon A^{p}(V)\rightarrow A$ to an abelian variety $A$ is called {\it regular} 
if for any smooth connected projective variety $S$ with a base point $s_{0}$ and for any codimension $p$ cycle $\Gamma$ on $S\times V$,
the composition 
\[S\rightarrow A^{p}(V) \rightarrow A, s\mapsto \phi(\Gamma_{*}(s-s_{0}))\]
is a morphism of algebraic varieties (this definition goes back to the work of Samuel \cite{Sa}).
An important example of such homomorphisms is the following.
We consider the Abel-Jacobi map
\[AJ^{p}\colon CH^{p}(V)_{\hom}\rightarrow J^{p}(V),\]
where $CH^{p}(V)_{\hom}\subset CH^{p}(V)$ is the subgroup of cycle classes homologous to zero,
and $$J^{p}(V)= H^{2p-1}(V, \C)/(H^{2p-1}(V,\Z(p)) + F^{p}H^{2p-1}(V, \C))$$ is the $p$-th Griffiths intermediate Jacobian
(see \cite[Section 12]{V0} for the definition and properties of the Abel-Jacobi maps).
Then the image $J^{p}_{a}(V)\subset J^{p}(V)$ of the restriction of the Abel-Jacobi map $AJ^{p}$ to $A^{p}(V)$ is an abelian variety, and the induced map
\[
\psi^{p}\colon A^{p}(V)\rightarrow J^{p}_{a}(V),
\]
which we also call Abel-Jacobi,
is regular \cite{Gr}\cite{Lie}.
A classical question of Murre \cite[Section 7]{M3}\cite[p. 132]{GMV} asks whether the Abel-Jacobi map $\psi^{p} \colon A^{p}(V)\rightarrow J^{p}_{a}(V)$ is {\it universal} among all regular homomorphisms $\phi\colon A^{p}(V)\rightarrow A$,
that is, whether every such $\phi$ factors through $\psi^{p}$.
This is known to hold for $p=1$ by the theory of the Picard variety, for $p=\dim V$ by the theory of the Albanese variety, and for $p=2$ as proved by Murre \cite{M1}\cite{M2} using the Merkurjev-Suslin theorem \cite{MS}.

Meanwhile, it was proved by Walker \cite{W} that the Abel-Jacobi map $\psi^{p}$ factors as
\[
\xymatrix{
 &J(N^{p-1}H^{2p-1}(V, \Z(p)))\ar[d]^-{\pi^{p}}\\
 A^{p}(V)\ar[ur]^-{\widetilde{\psi}^{p}}\ar[r]_-{\psi^{p}}& J^{p}_{a}(V)\\
},
\]
where 
$J(N^{p-1}H^{2p-1}(V, \Z(p)))$ is the intermediate Jacobian for the mixed Hodge structure given by the coniveau filtration $N^{p-1}H^{2p-1}(V, \Z(p))$ \cite{BlO},
$\pi^{p}$ is a natural isogeny,
and $\widetilde{\psi}^{p}$ is a surjective regular homomorphism.
If the Abel-Jacobi map $\psi^{p}$ is universal, then the kernel
\[\Ker(\pi^{p})= \Coker\left(H^{2p-1}(V,\Z(p))_{\tors}\rightarrow (H^{2p-1}(V,\Z(p))/N^{p-1}H^{2p-1}(V,\Z(p)))_{\tors}\right)\]
is trivial.
In other words, the sublattice
\[
N^{p-1}H^{2p-1}(V,\Z(p))/\tors\subset H^{2p-1}(V,\Z(p))/\tors
\] 
is primitive.

We recall the main theorem of the paper \cite{S}.

\begin{thm}[\cite{S}, Theorem 1.3]\label{tS}
Let $W$ be a smooth projective variety such that $CH_{0}(W)$ is supported on a surface and the inclusion
\[
H^{4}_{\alg}(W, \Z)/\tors\subsetneq Hdg^{4}(W,\Z)/\tors
\]
is strict.
Then there exists an elliptic curve $E$ such that the sublattice
\[
N^{2}H^{5}(W\times E, \Z(3))/\tors \subset H^{5}(W\times E, \Z(3))/\tors
\]
is not primitive.
\end{thm}

Now we prove that the Abel-Jacobi map is not universal in general.
This settles Murre's question.

\begin{cor}\label{c'}
Let $\mathcal{X}$ be the pencil of Enriques surfaces of Theorem \ref{t'}.
Then there exists an elliptic curve $E$ such that 
the Abel-Jacobi map
\[
\psi^{3}\colon A^{3}(\mathcal{X}\times E)\rightarrow J_{a}^{3}(\mathcal{X}\times E)
\]
is not universal: 
it factors as
\[
\xymatrix{
 &J(N^{2}H^{5}(\mathcal{X}\times E, \Z(3)))\ar[d]^-{\pi^{3}}\\
 A^{3}(\mathcal{X}\times E)\ar[ur]^-{\widetilde{\psi}^{3}}\ar[r]_-{\psi^{3}}& J^{3}_{a}(\mathcal{X}\times E)\\
},
\]
where 
the Walker map $\widetilde{\psi}^{3}$ is surjective regular, and 
the natural isogeny $\pi^{3}$ has non-zero kernel, or equivalently, 
the sublattice 
\[
N^{2}H^{5}(\mathcal{X}\times E,\Z(3))\subset H^{5}(\mathcal{X}\times E,\Z(3))
\] 
is not primitive.
\end{cor}
\begin{rem}
The Walker map $\widetilde{\psi}^{3}$ in the statement is universal by \cite[Theorem 1.1]{S}.
\end{rem}
\begin{rem} 
In fact, we have $N^{2}H^{5}(\mathcal{X}\times E, \Q)=H^{5}(\mathcal{X}\times E, \Q)$ as a consequence of decomposition of the diagonal \cite{BS}.
In other words, $J^{3}_{a}(\mathcal{X}\times E)=J^{3}(\mathcal{X}\times E)$ (see \cite[Lemma 4.3]{M3}).
\end{rem}
\begin{proof}[Proof of Corollary \ref{c'}]
We have $CH_{0}(\mathcal{X})=\Z$ by Lemma \ref{l2}.
Moreover, 
the cohomology group $H^{4}(\mathcal{X},\Z)$ is torsion-free
and 
the inclusion $H^{4}_{\alg}(\mathcal{X},\Z)\subsetneq Hdg^{4}(\mathcal{X},\Z)$ is strict by Theorem \ref{t'}.
Now the assertion follows by applying Theorem \ref{tS} to $W=\mathcal{X}$.
The proof is complete.
\end{proof}

Finally, we explain how to produce counterexamples to Murre's question in higher dimensions and for other values of $p$.
We take $\mathcal{X}$ and $E$ as in Corollary \ref{c'}, and let $d\geq 4$.
Then, on the $d$-fold $\mathcal{X}\times E\times \P^{d-4}$, for all $3\leq p\leq d-1$, the sublattice
\[N^{p-1}H^{2p-1}(\mathcal{X}\times E\times \P^{d-4}, \Z(p))\subset H^{2p-1}(\mathcal{X}\times E\times \P^{d-4},\Z(p))\]
is not primitive
(this follows from the formula \cite[Theorem 3.1]{Ba} for the Bloch-Ogus spectral sequence \cite{BlO} under taking the product with a projective space).
In particular, for all $3\leq p\leq d-1$, the Abel-Jacobi map
\[\psi^{p}\colon A^{p}(\mathcal{X}\times E\times \P^{d-4})\rightarrow J^{p}_{a}(\mathcal{X}\times E\times \P^{d-4})\]
is not universal.

\end{document}